\documentclass[oneside]{amsart}
\usepackage{geometry}

\usepackage[utf8]{inputenc} 
\usepackage{amsmath,amsfonts,amssymb,amsthm}
\usepackage{color,xcolor}

\usepackage[all]{xy}
\entrymodifiers={+!!<0pt,\fontdimen22\textfont2>}

\usepackage{hyperref}	

\theoremstyle{plain}
\newtheorem{thm}{Theorem}[section]
\newtheorem{lem}[thm]{Lemma}
\newtheorem{prop}[thm]{Proposition}
\newtheorem{cor}[thm]{Corollary}

\theoremstyle{definition}
\newtheorem{defn}[thm]{Definition}

\theoremstyle{remark}
\newtheorem{rem}[thm]{Remark}

\numberwithin{equation}{section}

\DeclareMathOperator\id{id}
\DeclareMathOperator\im{im}

\DeclareMathOperator\End{End}

\newcommand{\Pin}{\mathsf{Pin}}
\newcommand{\GL}{\mathsf{GL}}

\newcommand{\Biv}{\operatorname{Biv}}

\newcommand{\FF}{\mathbb{F}}
\newcommand{\CC}{\mathbb{C}}

\newcommand{\ZZ}{\mathbb{Z}}

\newcommand{\U}{\mathcal{U}}
\renewcommand{\O}{\mathcal{O}}

\newcommand\kfrak{\mathfrak{k}}
\newcommand\pfrak{\mathfrak{p}}
\newcommand{\gfrak}{\mathfrak{g}}

\newcommand{\so}{{\mathfrak{so}}}

\newcommand{\grA}{{\overline A}}
\newcommand{\grd}{{\overline d}}
\newcommand{\grD}{\overline{D}}
\newcommand{\grx}{\overline{x}}
\newcommand{\gry}{{\overline{y}}}

\newcommand{\tiH}{{\tilde H}}
\newcommand{\tih}{{\tilde{h}}}
\newcommand{\tiG}{{\tilde G}}
\newcommand{\tig}{{\tilde g}}

\newcommand{\itriv}{_{\textnormal{triv}}}
\newcommand{\idet}{_{\textnormal{det}}}
\newcommand{\ieven}{^{\textnormal{even}}}
\newcommand{\iodd}{^{\textnormal{odd}}}

\newcommand{\inv}{^{-1}}
\newcommand{\tensor}{\otimes}
\newcommand{\eps}{\varepsilon}
\newcommand{\DeltaC}{{\Delta_C}}
\newcommand{\AC}{{A\tensor C}}
\newcommand{\tri}[1][\cdot,\cdot]{\langle #1 \rangle}

\newcommand{\hiddensubsection}{\stepcounter{subsection}
\medskip
\noindent\thesubsection\hspace{1em}}

\begin{document} 

\title{Barbasch-Sahi algebras and Dirac cohomology}
\author{Johannes Flake}

\begin{abstract} 
We define a class of algebras which are distinguished by a PBW property and an orthogonality condition, and which we call \emph{Hopf-Hecke algebras}, since they generalize the Drinfeld Hecke algebras defined by Drinfeld in \cite{Dr}. In the course of studying the orthogonality condition and in analogy to the orthogonal group we show the existence of a \emph{pin cover} for cocommutative Hopf algebras over $\CC$ with an orthogonal module or, more generally, pointed cocommutative Hopf algebras over a field of characteristic $0$ with an orthogonal module.

Following the suggestion of Dan Barbasch and Siddhartha Sahi, we define a Dirac operator and Dirac cohomology for modules of Hopf-Hecke algebras, generalizing those concepts for connected semisimple Lie groups, graded affine Hecke algebras and symplectic reflection algebras. Using the pin cover, we prove a general theorem for a class of Hopf-Hecke algebras which we call \emph{Barbasch-Sahi algebras}, which relates the central character of an irreducible module with non-vanishing Dirac cohomology to the central characters occurring in its Dirac cohomology, generalizing a result called ``Vogan's conjecture'' for connected semisimple Lie groups which was proved in \cite{HP}, analogous results for graded affine Hecke algebras in \cite{BCT} and for symplectic reflection algebras and Drinfeld Hecke algebras in \cite{Ci}.
\end{abstract}

\maketitle

\setcounter{tocdepth}{2}
\tableofcontents

\section{Introduction}
\hiddensubsection Let $\FF$ be a field of characteristic $0$. All vector spaces and tensor products are over $\FF$, all modules are left modules, and we will be working with finite-dimensional modules only.

We consider algebras $A$ of the form 
\[ A= (T(V)\rtimes H)/I_\kappa
\ ,
\]
where $H$ is a cocommutative Hopf algebra, $V$ is a finite-dimensional $H$-module, $T(V)$ is its tensor algebra, which is an $H$-module algebra, $T(V)\rtimes H$ is the semidirect/smash product, $\kappa:V\wedge V\to H$ is an $H$-linear map and $I_\kappa$ is the ideal in $T(V)\rtimes H$ generated by elements of the form $v\tensor w-w\tensor v-\kappa(v,w)$ for $v,w\in V$. We impose two additional conditions on the discussed structures:

First, we require $A$ to have a \emph{PBW property}, i.e., to be a flat deformation of $S(V)\rtimes H$. This means we require the natural surjection of $S(V)\rtimes H$ to the associated graded algebra of $A$ (with respect to the filtration of $T(V)$) to be an isomorphism. The PBW property depends on the map $\kappa$ in ways which have been studied extensively in the literature, starting with the work of Drinfeld in \cite{Dr}, Braverman and Gaitsgory in \cite{BG} and of Etingof and Ginzburg in \cite{EG}. We refer to the survey \cite{SW-survey}, the article \cite{WW} and the recent article \cite{Kh} which contains many references in its introduction. The implications for our special case are studied in \cite{FS}.

Second, we require that $V$ be \emph{orthogonal}, i.e., it should have a non-degenerate $H$-invariant symmetric bilinear form $\tri$. In this situation, let $C=C(V)$ be the Clifford algebra defined by $V$ and $\tri$, that is, the algebra $T(V)/J_\kappa$, where $J_\kappa$ is the ideal generated by elements of the form $v\tensor w+w\tensor v-2\tri[v,w]$. Since $V$ is orthogonal, this is an $H$-module algebra, as well.

In this paper we call algebras of this form \emph{Hopf-Hecke algebras}. They are determined by the cocommutative Hopf algebra $H$, its orthogonal module $V$ and the map $\kappa$. For a fixed Hopf-Hecke algebra $A$, let $(v_i)_i$, $(v^i)_i$ be a pair of dual bases of $V$ with respect to $\tri$.  We define a \emph{Dirac element} 
\[ D := \sum_i v_i \tensor v^i
\]
in $A\tensor C$, where we identify elements in $V$ with elements in $A$ and $C$, respectively. 

Let $S$ be a spin module of $C$. If $M$ is an $A$-module, then $M\tensor S$ is an $A\tensor C$-module, so in particular, $D$ acts on $M\tensor S$. We define the \emph{Dirac cohomology} of $M$ as
\[ H^D(M) := \ker D/(\ker D\cap\im D)
\ . 
\]

In this paper we show that, if $H$ is a pointed cocommutative Hopf algebra over $\FF$ with an orthogonal module $V$, then $H$ has what we call a \emph{pin cover} with respect to $V$, that is, there are a pointed cocommutative Hopf algebra $\tiH$, a surjective Hopf algebra map $\pi:\tiH\to H$ and an algebra map $\gamma:\tiH\to C$ such that $\pi$ splits as coalgebra map in a certain way, $\gamma$ preserves a certain $\ZZ_2$-gradation, and $\pi$, $\gamma$ and the action of $H$ on $V$ satisfy a compatibility condition (\ref{prop-pin-cover}). This generalizes the pin cover of a group acting on a vector space by orthogonal operators.

We define $\DeltaC:=(\pi\tensor\gamma)\circ\Delta:\tiH\to H\tensor C$ and $H':=\tiH/\ker\DeltaC$. We will see that $\tiH$ and also $H'$ act on the Dirac cohomology $H^D(M)$ via $\DeltaC$. We say that $D$ satisfies the \emph{Parthasarathy condition}, if 
\[ D^2 \in Z(A\tensor C)+\DeltaC(\tiH\ieven)
\ , 
\]
where $Z(\AC)$ is the center of $\AC$.

If a Hopf-Hecke algebra $A$ has a Dirac element $D$ that satisfies the Parthasarathy condition, we call $A$ a \emph{Barbasch-Sahi algebra} (since the idea of studying Dirac cohomology in such a general situation was developed by Dan Barbasch and Siddhartha Sahi). We prove:
\begin{prop}[\ref{algebraic-Vogans-conjecture}] If $A$ is a Barbasch-Sahi algebra, then there is an algebra map $\zeta:Z(A)\to Z(H')$ such that for every $z\in Z(A)$ there exists an element $b\in\AC$ such that 
\[ z\tensor1 = \DeltaC(\zeta(z))+bD+Db
\ .
\]
\end{prop}
In the context of connected semisimple Lie groups (\cite{HP}), this corresponds to the ``algebraic Vogan's conjecture'' and just as in that setting also here it implies the following (``Vogan's conjecture''):
\begin{thm}[\ref{Vogans-conjecture}] Assume $M$ is an $A$-module with central character $\chi$ and with non-zero Dirac cohomology $H^D(M)$ and let $(U,\sigma)$ be a non-zero $H'$-submodule of $H^D(M)$. Then $\chi=\sigma\circ\zeta$.
\end{thm}

As mentioned in the introduction of \cite{HP}, the conjectures were formulated by David Vogan in the fall of 1997 in a series of talks at the MIT Lie groups seminar.

\hiddensubsection Our definition of Hopf-Hecke algebra generalizes several situations discussed in the literature in which the Dirac cohomology and its connection with central characters have been studied.

In Huang's and Pand{\v{z}}i{\'c}'s setting (\cite{HP}), where $\gfrak$ is the complexification of a real semisimple Lie algebra and  $\gfrak=\kfrak\oplus\pfrak$ is its Cartan decomposition, $H=\U(\kfrak)$, $V=\pfrak$ and $\kappa$ comes from the Lie bracket in $\gfrak$ (in particular, its image lies in $\kfrak\subset H$). 

Ciubotaru studied the situation where $H$ is the group algebra of a group $W$ acting faithfully and orthogonally on $V$ and the map $\kappa$ is general (\cite{Ci}), thereby extending the results of \cite{BCT}. The resulting family of algebras was introduced by Drinfeld in \cite{Dr} and the maps $\kappa$ yielding the PBW property were studied in \cite{RS} (see also \cite{SW}). The graded affine Hecke algebras studied by Lusztig (\cite{Lu}), the setting in \cite{BCT}, are a special case, just as symplectic reflection algebras introduced by Etingof and Ginzburg in \cite{EG} and the rational Cherednik algebras $\mathsf{H}_{t,c}$ with parameters $t,c$. Dirac operators and a Dirac cohomology are introduced for them in \cite{Ci}, and it is proved that graded affine Hecke algebras and also rational Cherednik algebras at $t=0$ are Barbasch-Sahi algebras (\cite{BCT,Ci}). In these situations, an analogue of Vogan's conjecture is shown to hold. 

A new example, as shown in \cite{FS}, consists of infinitesimal Cherednik algebras as introduced in \cite{EGG}. Just as in the setting of \cite{HP}, the Hopf algebra $H$ is the enveloping algebra of a Lie algebra $\gfrak$, but in contrast to the situation there, the map $\kappa$ takes values in arbitrary degrees of the algebra filtration of $\U(\gfrak)$.

\hiddensubsection \label{generalizations} The presented ideas could be further generalized in several directions:

One could study algebras similar to Hopf-Hecke algebras, but allowing $\kappa$ to take values in $H\oplus V$, and try to define cubic Dirac operators for those algebras following Kostant's ideas in \cite{Ko-multiplets,Ko-cubic}. The PBW property for the special case where $H$ is a group algebra is discussed in \cite{SW-orbifold}.

The cocommutative Hopf algebra $H$ could be replaced by a \emph{cocommutative algebra} as considered in \cite{Kh} or by a quasi-triangular (quasi-)Hopf algebra. $S(V)$ could be replaced by a general Koszul algebra, as discussed in \cite{WW, WW2}.

It would be interesting to try to generalize the discussed concept for a suitable subclass of \emph{generalized Weyl algebras} as considered in \cite{Kh-Weyl, KT}.

In situations, where the Dirac operator does not satisfy the Parthasarathy condition and results analogous to Vogan's conjecture cannot be shown through a direct generalization of the proof in \cite{HP}, one could think of generalizations in the spirit of \cite[thm.~3.14]{Ci}.

\hiddensubsection The construction of a pin cover for pointed cocommutative Hopf algebra is section~\ref{sec-pin-cover}, Hopf-Hecke algebras are defined and their Dirac cohomology is studied in section~\ref{sec-Hopf-Hecke}, and the connection with central characters is made in section~\ref{sec-Vogans-conjecture}.

\hiddensubsection \textbf{Acknowledgments.} The key ideas of this paper, that one should study Dirac cohomology for Hopf-Hecke algebras and establish an analogue of Vogan's conjecture, originated with Dan Barbasch and Siddhartha Sahi and was discussed by them extensively. I would like to thank them both for their generosity in sharing these ideas. In addition, I would like to thank Siddhartha Sahi for his continued encouragement and assistance in the course of developing this program and bringing these ideas to fruition.

I would also like to thank Apoorva Khare and Chelsea Walton for their valuable feedback on an earlier version of this paper.


\section{Pin cover of a pointed cocommutative Hopf algebra} 
\label{sec-pin-cover}
We will review some basic Hopf algebra theory, for more information on Hopf algebras we refer to \cite{Mo}.

We will always denote the counit of a Hopf algebra $H$ by $\eps:H\to\FF$, the coproduct by $\Delta:H\to H\tensor H$ and the antipode by $S:H\to H$. Also, we will use Sweedler's notation: If $H$ is a Hopf algebra (or coalgebra, for that matter) with coproduct $\Delta:H\to H\tensor H$ and $h$ is an element of $H$, then we will write $h_{(1)}\tensor h_{(2)}$ for the coproduct $\Delta(h)$ in $H\tensor H$, which is not necessarily a pure tensor, but in general a sum of several pure tensors, which is implied by this notation. Note that expressions like $h_{(1)}\tensor h_{(2)}\tensor h_{(3)}$ for the twofold coproduct are unambiguous due to coassociativity of coalgebras. Finally, we recall that every Hopf algebra $H$ acts on itself by the $\emph{(left) adjoint action}$ 
\[ h\cdot k = h_{(1)} k Sh_{(2)}
\]
for $h,k\in H$.

\begin{defn} A coalgebra $C$ with coproduct $\Delta:C\to C\tensor C$ is called \emph{pointed} if every simple subcoalgebra of $C$ is one-dimensional.
An element $c\in C$ is called \emph{group-like} if $\Delta c=c\tensor c$ and \emph{primitive} if $\Delta c=1\tensor c+c\tensor 1$. The set of group-likes is denoted by $G(C)$, the set of primitives is denoted by $P(C)$.
\end{defn}

We recall from basic Hopf algebra theory that for every Hopf algebra $H$, $G(H)$ is a group with multiplication in $H$ and $P(H)$ is a Lie subalgebra of $H$ with the commutator.

\begin{defn} If $H$ is a Hopf algebra and $B$ is an $H$-module algebra, then the \emph{semidirect/smash product} $B\rtimes H$ is the algebra generated by $B$ and $H$ and the relation $hb=(h_{(1)} \cdot b) h_{(2)}$ for all $h\in H,b\in B$.
\end{defn}

We recall the following structure theorem: if $H$ is a cocommutative pointed Hopf algebra over a field $\FF$ of characteristic $0$, then $H=\U(P(H))\rtimes \FF[G(H)]$, where $\U(P(H))$ is the universal enveloping algebra of the Lie algebra of primitive elements in $H$ and $\FF[G(H)]$ is the group algebra of the group of group-likes in $H$. This applies in particular to any cocommutative Hopf algebra over $\FF=\CC$, because every cocommutative coalgebra over an algebraically closed field is pointed.

Motivated by pin covers of groups acting orthogonally on a module, we define an analogous concept for pointed cocommutative Hopf algebras. Just as for groups, we start by defining orthogonal actions.

\begin{lem} \label{lem-orthogonal} Let $H$ be a cocommutative Hopf algebra (with counit $\eps$, coproduct $\Delta$ and antipode $S$) and let $V$ be a finite-dimensional $H$-module with non-degenerate symmetric bilinear form $\tri$. We denote the action of $h\in H$ on $v\in V$ by $h \cdot v$. Consider the following conditions:
\begin{enumerate}
\item{$\tri[h_{(1)} \cdot v,h_{(2)} \cdot w]=\eps(h)\tri[v,w]$ for all $h\in H,v,w\in V$.}
\item{$\tri[h \cdot v,w]=\tri[v,S(h) \cdot w]$ for all $h\in H,v,w\in V$.}
\item{Group-like elements of $H$ act as orthogonal linear operators on $V$ and primitive elements of $H$ act as skew-adjoint linear operators on $V$ with respect to $\tri$.}
\end{enumerate}

Then $(1)\Leftrightarrow(2)\Rightarrow(3)$ and if $H$ is pointed (e.g., $\FF=\CC$), then even $(1)\Leftrightarrow(2)\Leftrightarrow(3)$.
\end{lem}

\begin{proof} Let $h$ be an element in $H$ and let $v,w$ be elements in $V$. If we assume (1) holds, then 
\[ \tri[h \cdot v,w]=\tri[h_{(1)} \cdot v,h_{(2)} \cdot (S(h_{(3)}) \cdot w)]
 = \eps(h_{(1)}) \tri[v,S(h_{(2)}) \cdot w]
 =\tri[v,S(h) \cdot w]
\ ,
\]
so (2) holds. If we assume that (2) holds, then
\[ \tri[h_{(1)} \cdot v,h_{(2)} \cdot w]=\tri[v,S(h_{(1)}) \cdot (h_{(2)} \cdot w)]=\eps(h)\tri[v,w]
\ ,
\]
which proves (1). So (1) and (2) are equivalent.

(3) follows from (2), because $S(h)=h^{-1}$ for a group-like element $h\in H$ and $S(h)=-h$ for a primitive element $h\in H$, respectively.

(2) follows from (3) if $H$ is pointed, because by the structure theorem, $H$ is generated as algebra by group-likes and primitives then, and $S$ is an anti-algebra map. 
\end{proof}

\begin{defn} In the situation of the lemma, the non-degenerate symmetric bilinear form $\tri$ on $V$ is called \emph{$H$-invariant} if it satisfies the above equivalent conditions. The finite-dimensional module $V$ of $H$ is called \emph{orthogonal} if $V$ has an $H$-invariant non-degenerate symmetric bilinear form. 
\end{defn}

Before showing the existence of a pin cover for pointed cocommutative Hopf algebras $H$ (including all cocommutative Hopf algebras over $\CC$), we define a $\ZZ_2$-grading for those $H$:

\begin{defn} \label{def-gradation-H} Let $H$ be a cocommutative pointed Hopf algebra and $V$ an orthogonal $H$-module. We define a $\ZZ_2$-gradation of $H$ as algebra denoted by $|\tih|\in\{0,1\}\simeq\ZZ_2$ or by $(-1)^{|\tih|}\in\{\pm1\}\simeq\ZZ_2$ for homogeneous elements $h\in H$ by assigning
$(-1)^{|h|}=\det h|_V$ for group-like elements $h$ and $(-1)^{|h|}=1$ for primitive-elements $h$. We denote the even and odd subspace by $H\ieven$ and $H\iodd$, respectively.
\end{defn}

\begin{rem} Since the gradation is defined using group-likes and primitives, $H\ieven$ and $H\iodd$ are subcoalgebras for every pointed cocommutative Hopf algebra $H$.
\end{rem}

In the following we will make use of results on the Clifford algebra $C(V)=C(V,\tri)$ which is the quotient of the tensor algebra $T(V)$ by the ideal generated by elements of the form $vw+wv-2\tri[v,w]$ for $v,w\in V$, the subgroup $\Pin(V)$ of $C(V)^\times$ generated by elements $v$ for $v\in V$ and $\tri[v,v]=1$, and the Lie subalgebra $\Biv(V)$ of $C(V)$ (with the commutator) generated by elements $vw$ for $v,w\in V$. We refer to \cite{HP-book, Me} for details. In particular, we will make use of the following lemmas: 

\begin{lem} \label{lem-Pin} There is group epimorphism $\Phi:\Pin(V)\to O(V)$ which maps the unit vector $v\in V$ viewed as element in $\Pin(V)$ to the reflection $\tau_v$ in the hyperplane orthogonally to $v$. For any $w\in V$, the element $\tau_v(w)$ equals $-vwv$ in $C(V)$.
\end{lem}

\begin{lem} \label{lem-Biv} There is a Lie algebra isomorphism $\phi:\Biv(V)\to\so(V)$ which maps a bivector $w_1w_2$ in $\Biv(V)$ to the map $v\mapsto [w_1w_2,v]$ in $\so(V)$, where the commutator is taken in $C(V)$ leaving $V\subset C(V)$ invariant.
The action of $O(V)$ on $T(V)$ induces an action of $O(V)$ on $\Biv(V)$, and $\phi$ is $O(V)$ equivariant, where $\so(V)$ is an $O(V)$-module under the action of $O(V)$ by conjugation.
\end{lem}

\begin{proof} The Lie algebra isomorphism is well-known.

The action of $O(V)$ on $T(V)$ induces an action on $C(V)$, since $V$ is an orthogonal module of $O(V)$, and the bivectors are invariant under this action.

To see that $\phi$ is $O(V)$-equivariant, we only have to consider reflections $\tau_v$ in $O(n)$. Assume $v$ is a unit vector in $V$ and $w_1w_2$ is a bivector in $C(V)$. Then for any $w\in V$,
\begin{align*}
 \phi(\tau_v \cdot (w_1w_2))(w)
 &= [\tau_v \cdot (w_1w_2),w]
 = [(-vw_1v)(-vw_2v),w]
 = [vw_1w_2v,w] \\
 &= vw_1 w_2vwvv - vvwvw_1w_2v
 = -v[w_1 w_2,-vwv]v
 = \tau_v \circ \phi(w_1 w_2)\circ \tau_v\inv (w)
\ ,
\end{align*}
as desired. 
\end{proof}

\begin{defn} Let $H,\tiH$ be Hopf algebras and $\pi:\tiH\to H$ a Hopf algebra epimorphism. We call $(\tiH,\pi)$ a \emph{double cover} of $H$ if $\tiH\cong H\oplus H$ as coalgebras and if $\pi$ restricts to an isomorphism of coalgebras on each copy of $H$.
\end{defn}

\begin{prop} \label{prop-pin-cover} Let $H$ be a cocommutative pointed Hopf algebra and $V$ an orthogonal $H$-module. Then there is a pointed cocommutative Hopf algebra $\tiH$, a Hopf algebra epimorphism $\pi:\tiH\to H$ and an algebra homomorphism $\gamma:\tiH\to C(V)$ such that $(\tiH,\pi)$ is a double cover of $H$, $\gamma$ preserves the $\ZZ_2$-gradation and the equation
\begin{equation} \label{eq-pin-cover}
(-1)^{|\tih|} \pi(\tih) \cdot v = \gamma(\tih_{(1)}) v \gamma(S\tih_{(2)})
\end{equation} 
holds in $C(V)$ for $v\in V$ and homogeneous elements $\tih\in\tiH$ (where the $\ZZ_2$-gradation of $\tiH$ is the one defined in \ref{def-gradation-H} using the fact that $V$ is an orthogonal $\tiH$-module via $\pi$).
\end{prop}

\begin{proof}
By the structure theorem of cocommutative pointed Hopf algebras over a field of characteristic $0$, we know that $H=\U(L)\rtimes\FF[G]$ for the group $G=G(H)$, its group algebra $\FF[G]$, the Lie algebra $L=P(H)$ of primitive elements in $H$, the universal enveloping algebra $\U(L)$ and an action of $G$ on $L$ which makes $\U(L)$ an $\FF[G]$-module algebra.

Since $V$ is an orthogonal module, there is a group homomorphism $G\to O(V)$. This implies there is a pin cover $(\tiG,\pi_G,\gamma_G)$ of the group $G$, that is a double cover $(\tiG,\pi_G:\tiG\to G)$ of $G$ and a group homomorphism $\gamma_G:\tiG\to\Pin(V)$ such that
\begin{equation} \label{eq-pin-cover-G}
(-1)^{|\tig|} \pi_G(\tig) \cdot v
 = \gamma_G(\tig)v\gamma_G(\tig\inv)
\end{equation}
in $C(V)$ for all $\tig\in\tiG,v\in V$. More explicitly, $\tiG$, $\pi_G$ and $\gamma_G$ arise as the pullback of the map $\Phi$ from \ref{lem-Pin} and the representation map $G\to O(V)$: 
\[ \label{cd-tiG}
\xymatrix{
 \tiG 
 		\ar@{-->}[r]^-{\gamma_G} 
 		\ar@{-->>}[d]^-{\pi_G} &
 \Pin(V)~
 		\ar@{^{(}->}[r]
 		\ar@{>>}[d]^\Phi &
 C(V) \\
 G
 		\ar[r] &
 O(V)
} 
\ . 
\]

Also since $V$ is an orthogonal module, there is a Lie algebra homomorphism $L\to\so(V)$, which by \ref{lem-Biv} implies there is a Lie algebra homomorphism $\gamma_L:L\to\Biv(V)$ such that 
\begin{equation} \label{eq-pin-cover-gfrak}
x \cdot v = [\gamma_L(x),v]
\end{equation}
in $C(V)$ for all $x\in L,v\in V$. More explicitly, note that the identification $\End(V)\simeq V\tensor V$ induces a map $\so(V)\to V\wedge V$. Then $\gamma_L$ sends the skew-symmetric endomorphism of $V$ corresponding to $v\wedge w$ to $\tfrac14(vw-wv)$ in $C(V)$.

Now we define $\tiH:=\U(L)\rtimes\FF[\tiG]$, where $\tiG$ acts on $L$ through the projection $\pi_G$ and the action of $G$ on $L$. $\pi_G$ extends to a Hopf algebra epimorphism $\pi_G:\FF[\tiG]\to\FF[G]$ which extends to a Hopf algebra epimorphism $\pi:\tiH\to H$ if
\[ \tig \cdot x
 = \pi_G(\tig) x \pi_G(\tig\inv)
\]
in $H$ for all $\tig\in\tiG,x\in L$, but both sides are equal to $\pi_G(\tig) \cdot x$, so this is true. We also note that $\pi$ splits as coalgebra map, because it maps group-likes to group-likes and primitives to primitives, and by construction $(\tiH,\pi)$ is a double cover of $H$, because $(\tiG,\pi_G)$ was a double cover of $G$.

Similarly, $\gamma_G$ and $\gamma_L$ can be extended to an algebra homomorphism $\gamma:\tiH\to C(V)$ if
\[ \gamma_L(\tig \cdot x)
 = \gamma_G(\tig) \gamma_L(x) \gamma_G(\tig\inv)
\]
in $C(V)$ for all $\tig\in\tiG,x\in L$, which is what we are going to prove in the following. We consider a reflection $\tau_v\in O(V)$, where $v$ is a unit vector in $V$. Then for any $w_1,w_2\in V$,
\[ v w_1 w_2 v
 = (-v w_1 v) (-v w_2 v)
 = (\tau_v \cdot w_1) (\tau_v \cdot w_2)
 = (\tau_v \cdot w_1) (\tau_v \cdot w_2) 
\]
in $C(V)$, so if $\tig$'s action on $V$ equals the product of reflections $\tau_{v_1}\dots\tau_{v_k}$ with unit vectors $v_1,\dots,v_k$ and if $\gamma_L(x)=w_1 w_2$, then
\[ \gamma_G(\tig) \gamma_L(x)
 \gamma_G(\tig\inv)
 = (\tig \cdot w_1)(\tig \cdot w_2)
 = \tig \cdot \gamma_L(x)
\ ,
\]
where we use the action of $\tiG$ on $C(V)$ induced by the action of $\tiG$ on $T(V)$. So the map $\gamma$ above is well-defined if $\gamma_L$ is $\tiG$-equivariant. Let us write $\gamma_L=\phi\inv\circ\psi$ with the Lie algebra map $\psi:L\to\so(V)$ corresponding to the action of $L$ on $V$ and the Lie algebra isomorphism $\phi$ as in \ref{lem-Biv}. Now $\so(V)$ is a $\tiG$-module through the projection $\pi_G$ and the action of $G$ on $V$, and $\psi$ is $\tiG$-equivariant, because
\[ \tig \cdot \psi(x)
 = (v\mapsto \pi_G(\tig) \psi(x) \pi_G(\tig\inv)v)
 = \psi(\pi_G(\tig) \cdot x)
 = \psi(\tig \cdot x)
\ ,
\]
where the second equality is due to the fact that the actions of $G$ and $L$ on $V$ come from an action of the smash-product $H$ on $V$. We have also seen that $\phi$ is $O(V)$-invariant in \ref{lem-Biv}. So $\gamma_L$ is $\tiG$-equivariant, and $\gamma$ is well-defined.

We want to show equation \eqref{eq-pin-cover}. Now it is enough to verify it for a set of algebra generators of $\tiH$, but this is just the content of \eqref{eq-pin-cover-G} (for group-likes) and \eqref{eq-pin-cover-gfrak} (for primitives).

It can be easily checked that $\gamma$ preserves the $\ZZ_2$-gradation considering group-likes and primitives, and by definition the action of $\tiH$ on $V$ via $\pi$ is orthogonal.
\end{proof}

\begin{defn} \label{defn-pin-cover} Let $H$ be a pointed cocommutative Hopf algebra with an orthogonal module $V$. The triple $(\tiH,\pi,\gamma)$ constructed in \ref{prop-pin-cover} is called \emph{pin cover} of $H$ with respect to $V$.
For any pin cover $(\tiH,\pi,\gamma)$ of $(H,V)$ we define a \emph{diagonal map} 
\[ \DeltaC:\tiH\to H\tensor C(V),
\tih\mapsto\pi(\tih_{(1)})\tensor\gamma(\tih_{(2)})
\ . 
\]
We also define the quotient algebra \label{def-tiHprime}
\[ H':=\tiH/\ker\DeltaC
\ . 
\]
\end{defn}

\begin{rem} \newcommand{\tie}{{\tilde e}} \label{rem-pin-cover}
The diagonal map $\DeltaC$ is an algebra map, because $\Delta,\pi,\gamma$ are algebra maps. We note that $\DeltaC$ is never injective. Consider the identity element $1$ of $H$ which acts trivially on $V$. It can be identified with the identity element $e$ of the group of group-like elements of $H$. In the pin cover of this group, it corresponds to two elements $\tie, \tie^-$ such that, as elements in $C(V)$, $\tie=1$ and $\tie^-=-1$. Hence $\gamma(\tie)=1=-\gamma(\tie^-)$, and consequently,
\[ \DeltaC(\tie)=e\tensor 1=-\DeltaC(\tie^-)
\ ,
\]
so $\DeltaC(\tie+\tie^-)=0$.

However, we observe that the pin cover splits in certain cases. For instance, if $H$ is the universal enveloping algebra of a Lie algebra $\gfrak$, then the identity element $1$ is the only group-like element of $H$. Now the map $\iota:H\to\tiH$ which sends $h\mapsto h\tensor1$ in $H\tensor\FF[\ZZ_2]\simeq \tiH$ for all $h\in H$ is a Hopf algebra map such that $\pi\circ\iota=\id_H$, so $\pi:\tiH\to H$ splits as Hopf algebra map, and $H$ can be identified with the Hopf subalgebra $\iota(H)$ of $\tiH$.
\end{rem}

\begin{defn} We say that a pin cover $(\tiH,\pi,\gamma)$ \emph{splits}, if the epimorphism $\pi:\tiH\to H$ splits as Hopf algebra map.
\end{defn}

\begin{lem} \label{lem-pin-cover} Let $(\tiH,\pi,\gamma)$ be the pin cover of $(H,V)$. If the pin cover of the group $G=G(H)$ splits, then $\pi:\tiH\to H$ splits as Hopf algebra map. In particular, this is the case if $H$ is the universal enveloping algebra of a Lie algebra.

If $\pi$ splits as Hopf algebra map, then $H'\cong H$ as algebras.
\end{lem}

\begin{proof} The first part follows from the construction in \ref{prop-pin-cover} and the observations in \ref{rem-pin-cover}.

To show the second part, we consider $H$ as Hopf subalgebra of $\tiH$, so by restriction we have an algebra map $\gamma:H\to C(V)$. If $\pi\inv(h)=\{\tih,\tih^-\}$ for an element $h\in H$, then $\DeltaC(\tih+\tih^-)=0$. So $H'$ is isomorphic to $H/\ker(\DeltaC|_H)$. Consider $h\in H$ in the kernel of $\DeltaC$. Then $h_{(1)}\tensor\gamma(h_{(2)})=0$. We apply the map $H\tensor C(V)\to H\tensor C(V)$ which sends $h\tensor d\to h_{(1)}\tensor \gamma(Sh_{(2)})d$ for any $h\in H$, $d\in C(V)$. Then 
\[ 0 = h_{(1)}\tensor \gamma(Sh_{(2)} h_{(3)}) = h\tensor 1
\ ,
\]
so $\DeltaC$ is injective on $H$ and $H'\cong H$, as desired.
\end{proof}

We drop the condition of pointedness and consider the special class of modules $V$ of the form $W\oplus W^*$.

\begin{lem} Let $H$ be a cocommutative Hopf algebra and $W$ a finite-dimensional $H$-module. Then there is an $H$-invariant non-degenerate symmetric bilinear form $\tri$ on $V=W\oplus W^*$, i.e. $V$ is an orthogonal $H$-module.
\end{lem}

\begin{proof} Let $(\cdot,\cdot):W^*\tensor W\to\FF$ be the natural pairing of $W^*$ and $W$. By definition of the contragredient action of $H$ on $W^*$, this pairing is $H$-invariant. We define a non-degenerate symmetric bilinear form $\tri$ on $V$ by $\tri[x+y,x'+y']=(x,y')+(x',y)$. This form is $H$-invariant, as well, so $V$ is an orthogonal $H$-module. 
\end{proof}

\begin{prop} \label{pin-cover-w-w} Let $H$ be a cocommutative Hopf algebra and $W$ a finite-dimensional $H$-module. Consider the orthogonal $H$-module $V=W\oplus W^*$. Then there is an algebra map $\gamma:H\to C(V)$ such that 
\[ h \cdot v = \gamma(h_{(1)}) v \gamma(Sh_{(2)})
\]
in $C(V)$ for all $h\in H$, $v\in V$.

If $H$ is pointed, then all elements of $H$ are even with respect to the $\ZZ_2$-gradation introduced in \ref{def-gradation-H} 
\end{prop}

\begin{proof} With respect to $\tri$ as in the lemma, $W\subset V$ is a maximal isotropic subspace. So $S:=C(V)/C(V)W$ is the spin module of $C(V)$ by left-multiplication and $C(V)=\End(S)$ (cp. \cite[3.2.2]{Me}). We observe that the action of $H$ on $T(V)$ induces an action of $H$ on $C(V)$, because $\tri$ is $H$-invariant. Now this action of $H$ on $C(V)$ induces an action of $H$ on $S$, because $W$ is preserved by the action of $H$. Hence there is an algebra map $\gamma:H\to\End(S)=C(V)$.

Consider elements $h\in H$, $v,v_1,\dots,v_k\in V$. Then $\gamma(h_{(1)})v\gamma(Sh_{(2)})$ is an element in $C(V)$, so we can apply it to $[v_1\cdots v_k]\in S=C(V)/C(V)W$:
\begin{align*}
\gamma(h_{(1)}) v\gamma(Sh_{(2)}) [v_1\cdots v_k]
 &= [h_{(1)} \cdot (v (Sh_{(2)} \cdot (v_1\cdots v_k)))]
 = [h_{(1)} \cdot (v (Sh_{(2)} \cdot v_1) \cdots (Sh_{(k+1)} \cdot v_k))] \\
 &= [(h_{(1)} \cdot v) ((h_{(2)} Sh_{(3)}) \cdot v_1) \cdots (h_{(2k)} Sh_{(2k+1)} ) \cdot v_k))]
 = [(h \cdot v) v_1\cdots v_k]
\end{align*} 
(using cocommutativity), as desired.

Now if $H$ is pointed and if a group-like $h\in H$ acts by the matrix $A$ on $W$ in a certain basis, its contragredient action on $W^*$ is described by the matrix $A\inv$ in the dual basis, so the determinant of $h$ acting on $V$ is always $1$, and hence all elements in $H$ are even.
\end{proof}

\begin{rem} We note that this construction yields a (split) pin cover with $\tiH=H\oplus H$ and $H'=H$ even if $H$ is not pointed.
\end{rem}
\section{Dirac cohomology for Hopf-Hecke algebras}
\label{sec-Hopf-Hecke}

\subsection{Hopf-Hecke algebras}
\label{subsec-Hopf-Hecke}
We fix a cocommutative Hopf algebra $H$ with finite-dimensional module $V$.

We are going to define what we mean by Hopf-Hecke algebras. Since $H$ is a cocommutative Hopf algebra acting on $V$, the tensor algebra $T(V)$ is an $H$-module algebra. The semidirect/smash product $T(V)\rtimes H$ is the algebra generated by $T(V)$ and $H$ and the relation
\[ hv = (h_{(1)} \cdot v) h_{(2)}
\]
for all $h\in H, v\in V$.

\begin{defn} Let $\kappa:V\wedge V\to H$ be a linear map. We denote by $I_\kappa$ the ideal of $T(V)\rtimes H$ generated by elements of the form $vw-wv-\kappa(v\wedge w)$ for $v,w\in V$. The $H$-module algebra 
\begin{equation}
A=A_\kappa:=(T(V)\rtimes H)/I_\kappa
\end{equation}
is called \emph{Hopf-Hecke algebra} of $(H,V,\kappa)$ if $V$ is an orthogonal module and if it satisfies the \emph{PBW property}, that is, if it is a \emph{flat deformation} of $S(V)\rtimes H$.
\end{defn}

In other words, let $\grA$ be the associated graded algebra of $A$ with respect to the filtration of the tensor factor $T(V)$. Now $A$ satisfies the PBW property if the natural surjection from $S(V)\rtimes H$ to $\grA$ is an isomorphism.

\begin{rem} This definition is closely related to the definition of continuous Hecke algebras in \cite{EGG}. More precisely, if $G$ is a reductive algebraic group and $\gfrak$ its Lie algebra, then the Hopf algebra $H=\U(\gfrak)\rtimes\FF[G]$ can be viewed as subalgebra of the algebra of algebraic distributions $\O(G)^*$ on $G$. Replacing $H$ with $\O(G)^*$ in the definition above and dropping the requirement of orthogonality for $V$ yields the definition of continuous Hecke algebras in the sense of \cite{EGG}.

A special case of our definition is the situation where $H$ is the group algebra of a finite group $G$. In this context, every module $V$ is orthogonal, and the class of algebras $A_\kappa$ has been studied in \cite{Dr, RS, SW}. 
\end{rem}

\begin{rem} As noted by Apoorva Khare, one could replace the Hopf algebra $H$ in the definition by a cocommutative algebra in the sense of \cite{Kh}, that is an algebra $H$ with a cocommutative coassociative coproduct $\Delta:H\to H\tensor H$ which is an algebra map, but not necessarily with counit and antipode. The orthogonality condition could still be formulated using \ref{lem-orthogonal}, and the PBW property for the resulting algebra $A=A_\kappa$ was studied in \cite{Kh}. 

However, for our purposes the existence of counit and antipode seem crucial, for instance in the proof of \ref{kernel-grd}.
\end{rem}

\subsection{Dirac elements}
We fix a cocommutative Hopf algebra $H$ with orthogonal (in particular, finite-dimensional) module $V$, with a linear map $\kappa:V\wedge V\to H$ such that $A=A_\kappa$ is a Hopf-Hecke algebra.  Since $V$ is fixed, we use the shorthand $C$ for the Clifford algebra $C(V)$. 

Before we focus on the case of a pointed cocommutative Hopf algebra $H$ using the pin cover constructed earlier, we define Dirac elements and Dirac cohomology for general Hopf-Hecke algebras.

\begin{defn} \label{defn-Casimir-Dirac} Let $(v_k)_k$, $(v^k)_k$ be a pair of orthogonal bases of $V$. Then the \emph{Casimir element} $\Omega$ and the \emph{Dirac element} $D$ are defined to be
\begin{equation} \label{eq-Casimir}
\Omega:= \sum_k v_k v^k
\in A
\ ,
\qquad
D:= \sum_k v_k \tensor v^k
\in A\tensor C
\ . 
\end{equation}
\end{defn}

\begin{lem} The Casimir and the Dirac element are independent of the choice of dual bases, and \label{square-of-D} 
\[ 
 D^2 = \Omega\tensor1
   + \tfrac12 \sum_{k<l} \kappa(v_k,v_l) \tensor [v^k, v^l]
   \ ,
\]
where the commutator is taken in $C$.
\end{lem}

\begin{proof} We consider the identity map on $V$ and we identify $V\simeq V^*$ using $\tri$, then $\Omega$ is the element in $V\tensor V$ viewed as element in $A$ corresponding to the identity map under the identification $\End(V)\simeq V\tensor V^*\simeq V\tensor V$ for any choice of dual bases. 

Since $D$ is just the image of $\id_V$ in $\AC$, it is independent of the choice of dual bases, as well. 

To verify the formula for $D^2$, we compute
\begin{align*}
D^2 
 &= \sum_{k,l} v_k v_l \tensor v^k v^l
  = \tfrac12 \sum_{k,l} v_k v_l \tensor (v^k v^l + 2\tri[v^k,v^l] - v^l v^k) \\
 &= \sum_{k,l} v_k v_l \tensor \tri[v^k,v^l]
  + \tfrac12 \sum_{k,l} [v_k,v_l] \tensor v^k v^l
  = \sum_k v_k v^k \tensor 1
  + \tfrac12 \sum_{k<l} [v_k,v_l] \tensor [v^k, v^l] 
\\
 &= \Omega\tensor1
   + \tfrac12 \sum_{k<l} \kappa(v_k,v_l) \tensor [v^k, v^l]
\ . 
\end{align*}
\end{proof}

Also for a general Hopf-Hecke algebra, we can define the concept of Dirac cohomology. Let $S$ be a fixed spin module of $C$. Let $M$ be an $A$-module. Then $M\tensor S$ is an $A\tensor C$-module, and in particular the Dirac operator acts on $M\tensor S$. 
\begin{defn} \label{defn-Dirac-cohomology} We define the \emph{Dirac cohomology of $M$} as
\[ H^D(M):=\ker D / (\im D\cap\ker D)
\ . 
\]
\end{defn}

From now on let us assume $H$ is pointed (which is guaranteed over $\FF=\CC$), so there is the pin cover $(\tiH,\pi,\gamma)$ of $(H,V)$.

\begin{defn}
We define an action of $\tiH$ on $\AC$ by 
\[ \tih\cdot x 
 := \DeltaC(\tih_{(1)}) x \DeltaC(S\tih_{(2)})
\] 
for all $\tih\in\tiH, x\in\AC$.
\end{defn}

\begin{lem} The above indeed defines an action of $\tiH$ on $\AC$ which makes $\AC$ a $\tiH$-module algebra.
\label{DeltaC-tiH-linear} With respect to the adjoint action of $\tiH$ on itself and the defined action, $\DeltaC$ is $\tiH$-linear and for all $x\in\AC,\tih\in\tiH$: \label{commutation-DeltaC} 
\[ \DeltaC(\tih) x = (\tih_{(1)}\cdot x) \DeltaC(\tih_{(2)})
\ . 
\] 
\end{lem}

\begin{proof} The map defines an action, because $\DeltaC$ and the coproduct of $\tiH$ are algebra maps and the antipode $S$ of $\tiH$ is an anti-algebra map. It yields a module algebra, because for all $\tih\in\tiH$ and for all $x,x'\in\AC$:
\[ \DeltaC(\tih_{(1)})xy\DeltaC(S\tih_{(2)})
 = \DeltaC(\tih_{(1)})x\DeltaC(S\tih_{(2)})
 	\DeltaC(\tih_{(3)})y\DeltaC(S\tih_{(4)})
\ . 
\] 

To see that $\DeltaC$ is $\tiH$-linear, note that for all $\tih,h'\in\tiH$:
\[ \DeltaC(\tih_{(1)} h' S\tih_{(2)})
 = \DeltaC(\tih_{(1)}) \DeltaC(h') \DeltaC(S\tih_{(2)})
 = \tih\cdot \DeltaC(h')
\ ,
\]
because $\pi$ and $c$ are algebra maps, where the order of indices in Sweedler's notation is arbitrary, because $\tiH$ is cocommutative.

For the last assertion, we compute
\[ \DeltaC(\tih) x 
 = \DeltaC(\tih_{(1)}) x \DeltaC(S\tih_{(2)}) \DeltaC(\tih_{(3)})
 = (\tih_{(1)}\cdot x) \DeltaC(\tih_{(2)})
\ . 
\]
\end{proof}

\begin{lem} \label{commutation-D-tih}  \label{commutation-D-rho} For all $\ZZ_2$-homogeneous $\tih\in\tiH$:
\[ \tih\cdot D=(-1)^{|\tih|} \eps(\tih) D 
\qquad\textnormal{and}\qquad
\DeltaC(\tih)D = (-1)^{|\tih|} D\DeltaC(\tih)
\ . 
\] 
\end{lem}

\begin{proof} For the first equation, we calculate (invoking \ref{prop-pin-cover})
\[ \tih\cdot D
 = \sum_i \pi(\tih_{(1)}) v_i \pi(S\tih_{(2)}) 
 		\tensor \gamma(\tih_{(3)}) v^i \gamma(S\tih_{(4)})
 = (-1)^{|\tih|}
	 	\sum_i \pi(\tih_{(1)}) \cdot v_i 
 		\tensor \pi(\tih_{(2)}) \cdot v^i
 = (-1)^{|\tih|}\eps(\tih) D
\]
for a homogeneous element $\tih\in\tiH$, because $\sum_i v_i\tensor v^i$ in $V\tensor V$ is $H$-invariant (since $\id_V$ is $H$-linear), and hence also $\tiH$-invariant. We recall that by definition the homogeneous $\ZZ_2$-graded subspaces of $H$, and hence of $\tiH$, are subcoalgebras.

Finally, we use this result to prove the second equation:
\[ D \DeltaC(\tih) 
 = (\tih_{(1)}\cdot D)\DeltaC(\tih_{(2)})
 = (-1)^{|\tih|} D\DeltaC(\tih) 
 \ . 
\]
\end{proof}

\subsection{The map \texorpdfstring{$d$}{d}} In order to investigate the cohomology of the Dirac element acting on suitable modules, we first study the cohomology of certain differentials. Our strategy is to directly generalize the proofs in \cite{HP, BCT, Ci}, and our most important technique are filtrations of algebras and associated graded algebras. 

\begin{lem} \label{graded-derivation} Let $R$ be a $\ZZ_2$-graded algebra and denote by $(-1)^{|x|}\in\{\pm 1\}$ the degree of a homogeneous element $x\in R$. Then for any odd $r\in R$, the map $f:R\to R$ defined on homogeneous elements $x$ by
\[ x\mapsto rx-(-1)^{|x|}xr
\]
is an odd ($\ZZ_2$-graded) derivation, and $f^2=[r^2,\cdot]$.
\end{lem}

\begin{proof} For any $x,y\in R$:
\[ f(xy)
 = rxy-(-1)^{|xy|}xyr
 = rxy-(-1)^{|x|}xry+(-1)^{|x|}xry-(-1)^{|x|+|y|}xyr
 = f(x)y+(-1)^{|x|}xf(y)
\ . 
\]

Also, since $|f(x)|\equiv |x|+1$ mod $2$,
\[ f^2(x)
 = rf(x)-(-1)^{|x|+1}f(x)r
 = r^2x-(-1)^{|x|}rxr+(-1)^{|x|}rxr+(-1)^{2|x|+1}xr^2
 = [r^2,x]
\ . 
\]
\end{proof}

\begin{defn} We define $d:\AC\to\AC$ to be the map which sends 
\begin{equation}
x\mapsto Dx-(-1)^{|x|}xD
\end{equation}
for any $x$ in $\AC$ which is homogeneous with respect to the $\ZZ_2$-gradation $\AC$ inherits from the tensor factor $C$.

For any $\tiH$-module $X$ we define $X\itriv$ and $X\idet$ to be the subspace on which $\tih$ acts as $\eps(\tih)$ and as $(-1)^{|\tih|}\eps(\tih)$, respectively, for all $\tih\in\tiH$. For any space $Y$ and any map $f:X\to Y$, we define the restrictions $f\itriv:=f|_{X\itriv}$ and $f\idet:=f|_{X\idet}$.

\label{defn-C1} Furthermore, we say that $D^2$ satisfies condition the \emph{Parthasarathy condition} if
\[ D^2 \in Z(A\tensor C) + \DeltaC(\tiH\ieven)
 \subset \AC
\ . 
\]
If $A$ is a Hopf-Hecke algebra with a Dirac element $D$ satisfying the Parthasarathy condition, we call $A$ a \emph{Barbasch-Sahi algebra}.
\end{defn}

We will see that the Parthasarathy condition is responsible for the fact that suitable restrictions of $d$ are differentials. This parallels the classical situation in \cite{HP}, where a statement on the square of the Dirac operator due to Parthasarathy (\cite[prop.~3.1]{Pa}) is used. Other examples of Barbasch-Sahi algebras are graded affine Hecke algebras (\cite[thm.~3.5]{BCT}) and rational Cherednik algebras at $t=0$ (\cite[4.3, rem.~4.10]{Ci}). In \cite{FS} it is proved that infinitesimal Cherednik algebras of $\GL_n$ are another example.

\begin{prop} \label{kernel-dtriv} 
\begin{enumerate}
\item{ $d$ is an odd derivation such that $d^2=[D^2,\cdot]$.
}

\item{ $d\circ\DeltaC=0$.
}

\item{ $d$ maps $(\AC)\itriv$ to $(\AC)\idet$ and vice versa.
}

\item{ If additionally $D^2$ satisfies the Parthasarathy condition, then $d^2=0$ on both subspaces, so
\begin{equation} \label{eq-ker-dtriv-inclusion}
 \ker d\itriv \supset \im d\idet + \DeltaC(Z(\tiH))
\ . 
\end{equation}
}
\end{enumerate}
\end{prop}

\begin{proof} Part (1) follows from \ref{graded-derivation}, because $D$ is an odd element in $\AC$ by definition. 

Part (2) follows from \ref{commutation-D-rho}, since for all $\tih\in\tiH$:
\[ d(\DeltaC(\tih)) 
 = D\DeltaC(\tih) - (-1)^{|\tih|}\DeltaC(\tih) D
 = D\DeltaC(\tih) - D\DeltaC(\tih) = 0
\ . 
\]

Part (3) follows from \ref{commutation-D-tih} and the implication that for all all $\ZZ_2$-homogeneous $\tih\in\tiH$ and all $x\in A\tensor C$,
\[ \tih\cdot(Dx)
 = (\tih_{(1)} \cdot D) (\tih_{(2)}\cdot x)
 = (-1)^{|\tih|} (\eps(\tih_{(1)}) D) (\tih_{(2)} \cdot x)
 = (-1)^{|\tih|} D (\tih\cdot x)
\ . 
\]

For the last part we check that for all $\tih\in\tiH\ieven$ and all $x$ in $(\AC)\itriv$ or $(\AC)\idet$:
\[ [\DeltaC(\tih), x]
 = (\tih_{(1)}\cdot x) \DeltaC(\tih_{(2)})
  - x \DeltaC(\tih)
 = (1 - 1) x \DeltaC(\tih)
 = 0
 \ ,
\]
so $d^2=[D^2,\cdot]=0$ on the subspaces in question. 
\end{proof}

From now on we assume that $D^2$ satisfies the Parthasarathy condition. In the following we will use the graded homomorphism associated to $d$ to refine the statement \eqref{eq-ker-dtriv-inclusion} on the kernel of $d\itriv$.

We recall that $A$ is a quotient of $T(V)\rtimes H$ as algebras, and the latter one has a $\ZZ$-gradation stemming from $T(V)$. Hence $A$ is a filtered algebra and hence $\AC$ is a filtered algebra whose filtered subspaces we denote by $F_i$ for $i\geq-1$ (where $F_{-1}=0$). The associated graded algebra of $A$ is denoted by $\grA$, and we recall that by assumption it is isomorphic to $S(V)\rtimes H$. We denote the homogeneous graded components of $\grA\tensor C$ by $\grA_i\tensor C$ for $i\geq0$. Finally, we have the linear projections $\gamma_i: F_i\to F_i/F_{i+1}=\grA_i\tensor C$ for all $i\geq0$.

The Dirac operator by definition lives in $F_1$, hence $d$ is a filtered endomorphism of $\AC$ of degree $1$ and it induces a graded endomorphism $\grd$ of degree $1$ of $\grA\tensor C$, i.e. $\grd(\grA_i\tensor C)\subset\grA_{i+1}\tensor C$ for all $i$. Explicitly, for any $i\geq 0$ and any element $\grx\in\grA_i\tensor C$ which is homogeneous with degree $|\grx|$ with respect to the $\ZZ_2$-gradation (stemming from $C$), we can pick any $x\in F_i$ such that $\gamma_i(x)=\grx$. Then $|x|=|\grx|$ and 
\[ \grd(\grx)
 = \gamma_{i+1}(Dx-(-1)^{|x|}xD)
 = \gamma_{i+1}(Dx)-(-1)^{|\grx|}\gamma_{i+1}(xD)
 = \grD\grx-(-1)^{|\grx|}\grx\grD
\ ,
\] 
where $\grD=\gamma_1(D)$ is the image of $D$ in $\grA_1\tensor C$.

\begin{lem} \label{grd-graded-derivation} \label{grd-H-linear} $\grd$ is a graded derivation and $\grd^2=0$, so $\grd$ defines a differential on $\grA\tensor C$.
\end{lem}

\begin{proof} From the explicit formula and \ref{graded-derivation} we see that $\grd$ is a graded derivation and $\grd^2=[\grD,\cdot]$. Hence, to see that $\grd^2=0$ it only remains to show that there is a set of algebra generators which commute with $\grD$.

Now let $(v_k)_k$ and $(v^k)_k$ be a pair of dual bases of $V$. Then
\[ \grD^2
 = \sum_{k,l} v_k v_l \tensor v^k v^l
 = \tfrac12 \sum_{k,l} v_k v_l \tensor (v^k v^l+v^l v^k)
 = \sum_k v_k v^k \tensor 1
 = \Omega\tensor1
\ ,
\]
because $v_k$ and $v_l$ commute in $\grA$ and $v^k v^l+v^l v^k=2\tri[v^k,v^l]$.
 
Therefore, $\grd^2(1\tensor v)=0$ and also $\grd^2(v\tensor 1)=0$ for all $v\in V$, because $S(V)$ is commutative. We consider $h\in H$. Then $h\tensor 1\in\grA_0\tensor C$ and
\[ \grd^2(h\tensor1)
 = [\Omega, h] \tensor 1
 = (\Omega h - (h_{(1)}\cdot\Omega)h_{(2)}) \tensor 1
 = 0
\]
because $\Omega$ is $H$-invariant. Hence, $\grd^2$ annihilates a set of algebra generators, as desired.
\end{proof}

Next, we want to describe the cohomology of $\grd$.

\begin{lem} \label{ker-dprime} Let $d'$ be the restriction of $\grd$ to $S(V)\tensor C\subset \grA\tensor C$. Then $d'$ preserves $S(V)\tensor C$ and is a multiple of the differential in the Koszul complex. In particular,
\[ \ker d' = \im d' \oplus \FF(1\tensor 1)
\ . 
\]
\end{lem}

\begin{proof} Again let $(v_i)_i$, $(v^i)_i$ be a pair of dual bases of $V$. We consider the element $v,w\in V$ and the element $v\tensor w$ in $S(V)\tensor C$. Then 
\[ \grd(v\tensor w)
 = \sum_i v_i v \tensor v^i w 
   - (-1)^{1} v v_i \tensor w v^i
 = \sum_k v v_i \tensor (v^i w + w v^i)
 = 2 v w \tensor 1
\ ,
\]
which coincides with the Koszul differential up to a constant. Since elements of the form $v\tensor w$ generate $S(V)\tensor C$ as algebra and since both maps in questions are graded derivations, they coincide.
\end{proof}

In the following we will view the map $\DeltaC:\tiH\to H\tensor C\subset F_0$ as a map $\tiH\to\grA_0\tensor C$ by identifying $F_0$ with $\grA_0\tensor C$.

\begin{prop} \label{kernel-grd} The spaces $\im\grd$ and $\DeltaC(\tiH)$ have trivial intersection and we have
\[ \ker\grd = \im\grd \oplus \DeltaC(\tiH)
\ . 
\]
\end{prop}

\begin{proof} Since $\grd$ is a graded endomorphism with degree $1$, $\im\grd$ is contained in $\bigoplus_{i\geq1} \grA_i\tensor C$. However, $\DeltaC(\tiH)$ is contained in $\grA_0\tensor C$, so their intersection is trivial and the sum on the right-hand side is direct.

We have seen that $\grd^2=0$ and \ref{kernel-dtriv} implies that $\DeltaC(\tiH)$ is contained in $\ker\grd$, so the right-hand side of the identity is contained in the left-hand side. It remains to show the opposite inclusion.

We have seen that $\pi$ splits as coalgebra map, so we can pick a coalgebra map $\iota:H\to\tiH$ such that $\pi\circ\iota=\id_H$. Now consider the maps $f,g:H\tensor C\to H\tensor C$ defined by
\[
 f(h\tensor d) = h_{(1)} \tensor \gamma(\iota(h_{(2)})) d
 \ ,\qquad
 g(h\tensor d) = h_{(1)} \tensor d \gamma(S\iota(h_{(2)}))
\] 
for all $h\in H, d\in C$. We verify that
\[ g\circ f(h\tensor 1) 
	= h_{(1)}\tensor \gamma(\iota(h_{(2)})S\iota(h_{(3)}))
	= h_{(1)}\tensor \gamma(\iota(h_{(2)})_{(1)} S\iota(h_{(2)})_{(2)})
	= h_{(1)}\tensor \eps(\iota(h_{(2)}))	
	= h\tensor 1
\ . 
\] 
Note that due to cocommutativity, we do not have to worry about the order of indices in Sweedler's notation.

Now we observe that $\grA\tensor C=(S(V)\tensor C)(H\tensor 1)=(S(V)\tensor C)f(H\tensor 1)$, because $h\tensor 1=(1\tensor \gamma(S\iota(h_{(1)}))f(h_{(2)}\tensor 1)$ for all $h\in H$. So for any $x\in\ker\grd$ we can find $(h^k)_k$ from $H$ and $r_k\in S(V)\tensor C$ such that $x=\sum r_k f(h^k\tensor1)$, and we can pick the $(h^k)_k$ to be linearly independent. Now note that for all $h\in H$,
\[ f(h\tensor1)
 = h_{(1)}\tensor \gamma(\iota(h_{(2)}))
 = \pi(\iota(h_{(1)})) \tensor \gamma(\iota(h_{(2)}))
 = \pi(\iota(h)_{(1)}) \tensor \gamma(\iota(h)_{(2)}) 
 = \DeltaC(\iota(h))
 \ ,
\]
so $f(H\tensor1)\subset\DeltaC(\tiH)$, and in particular, $\grd$ is zero on $f(H\tensor1)$. Hence
\[ 0 = \grd(x)
 = \sum_k \grd(r_k) f(h^k\tensor1) 
\ . 
\]
But this implies 
\[ 0 = g(\grd(x))
 = \sum_k \grd(r_k) g\circ f(h^k\tensor1) 
 = \sum_k \grd(r_k) (h^k\tensor1) 
\ ,
\]
where we have exploited that $g$ is $C$-linear with respect to multiplication from the left.
Now the $(h^k)_k$ were chosen to be linearly independent, so the last equation implies $\grd(r_k)=0$ for all $k$, or $r_k\in\ker d'$, equivalently. Hence by \ref{ker-dprime}, $r_k=p_k+ \grd(q_k)$ for $p_k\in\FF$ and $q_k\in S(V)\tensor C$ . Thus
\[ x = \sum_k (p_k + \grd(q_k)) f(h^k\tensor 1)
 = \sum_k (p_k + \grd(q_k)) \DeltaC(\iota(h^k))
 = \sum_k \DeltaC(p_k\iota(h^k))
 	 + \grd(q_k \DeltaC(\iota(h^k))) 
\,
\]
which is in $\DeltaC(\tiH) + \im\grd$, as desired.
\end{proof}

\begin{cor} \label{kernel-grdtriv} The action of $\tiH$ on $\AC$ preserves its filtration, so $\grA\tensor C$ is an $\tiH$-module.

The direct sum $\ker\grd=\im\grd\oplus\DeltaC(\tiH)$ in \ref{kernel-grd} is a direct sum of $H$-submodules.

In particular,
\[ \ker \grd\itriv = \im\grd\idet \oplus \DeltaC(Z(\tiH))
\ . 
\]
\end{cor}

\begin{proof} The action of $\tiH$ on $\AC$ indeed preserves the filtration by definition, because the image of $\DeltaC$ is in degree $0$ if the filtration. We take $\tiH$-invariants of the previously found direct sum decomposition of $\ker\grd$. As $\DeltaC$ is $\tiH$-linear, its image is a submodule. We have seen that $D$ lies in $(\AC)\idet$, so for all $x\in F_i$,
\[ \tih\cdot\grd(\grx)
 = \gamma_{i+1}(\tih\cdot d(x))
 = \gamma_{i+1}((-1)^{|\tih|} d(\tih\cdot x))
 = (-1)^{|\tih|} \grd(\tih\cdot\grx)
\]
and in particular, the image of $\grd$ is a submodule of $\grA\tensor C$. Hence the direct sum decomposition is a decomposition as $\tiH$-modules and we get
\[ \ker\grd\itriv
 = (\ker\grd)^\tiH
 = (\im\grd)^\tiH\oplus\DeltaC(\tiH)^\tiH
\ . 
\] 
Using $\tiH$-linearity of $\DeltaC$ and the corresponding formula for $\grd$ again we obtain the desired result, because the center of a Hopf algebra $H$ is just the set of $H$-invariants. 
\end{proof}

\begin{lem} \label{lem-sum-direct} The intersection of the spaces $\im d\idet$ and $F_0$ is $0$.
\end{lem}

\begin{proof} 
This follows from our description of the kernel of the graded map $\grd$: Assume $x=d(a)$ for some $a\in(\AC)\idet$ and at the same time $x\in F_0$. We have to show that $x=0$. If $a=0$ then we are done, so we assume $a\neq0$. Let $i\geq0$ be the degree of $a$ with respect to the filtration, i.e. $a\in F_i\setminus F_{i-1}$. We choose $a$ such that $i$ is minimal. Then
\[ \grd(\gamma_i(a))
 = \gamma_{i+1}(x)
 = 0
 \ ,
\]
because x lies in degree $0$ of the filtration. Now we use \ref{kernel-grdtriv} (or \ref{kernel-grd}) to observe that $\gamma_i(a)=\grd(\gamma_{i-1}(b))$ for some $b\in F_{i-1}\cap(\AC)\itriv$ if $i\geq 1$, or $\gamma_i(a)$ lies in the image of $\DeltaC$ if $i=0$. In the first case, $\gamma_i(a-d(b))=0$, so $a':=a-d(b)\in F_{i-1}$. Now $a=d(b)+a'$ or $a\in\DeltaC(Z(\tiH))$. In the first case, $d(a)=d(a')$ would contradict the minimality of the degree of $a$, and in the second case, $d(a)=0$, which is what we aimed to show.
\end{proof}

We are now ready to refine \ref{kernel-dtriv}.
\begin{prop} \label{homology-of-d} We have
\[
\ker d\itriv = \im d\idet \oplus \DeltaC(Z(\tiH))
\ . 
\]
\end{prop}

\begin{proof} After \ref{kernel-dtriv} it suffices to see that the sum on the right-hand side is direct and that it contains $\ker d\itriv$. Now the sum is direct due to \ref{lem-sum-direct}.

Let $P:=\im d\idet\oplus\DeltaC(Z(\tiH))$. It remains to show that $\ker d\itriv\subset P$ and we will prove this by induction in the degree of the filtration on $\ker d\itriv$ inherited from its super space $A\tensor C$. 

First we note that $F_{-1}\cap\ker d\itriv=\{0\}\subset P$ trivially. Now we assume $F_{i-1}\cap\ker d\itriv\subset P$ for some $i\geq 0$. We consider an element $x\in(\AC)\itriv\cap(F_i\setminus F_{i-1})$ such that $d(x)=0$. Passing to the associated graded algebra yields
\[ \grd(\gamma_i(x))
  = \gamma_{i+1}(d(x))
  = 0
\ ,
\] 
and the projection $\gamma_i$ is (by definition of the $\tiH$-action on $\grA\tensor C$) $\tiH$-linear, so $\gamma_i(x)\in\ker\grd\itriv$ can be written as $\gamma_i(x)=\DeltaC(\tih)+\grd(\gry)$ for some $\tih\in Z(\tiH)$ and $\gry\in(\grA_{i-1}\tensor C)\idet$ by \ref{kernel-grd}. We even have
\[ \gamma_i(x)=\gamma_i(\DeltaC(\tih))+\grd(\gry)
\ ,
\]
because the image of $\DeltaC$ lies in $\grA_0\tensor C$.

Next we need to have a section $\eta:\grA_{i-1}\tensor C\to F_{i-1}$ of $\gamma_{i-1}$ as $\tiH$-modules. Let $\eta':S^{i-1}(V)\to T(V)$ be the symmetrization map. We note that $A\tensor C$ and $\grA\tensor C$ are both isomorphic to $S(V)\tensor H\tensor C$ as vector spaces and we define $\eta:=\eta'\tensor\id\tensor\id$ as linear map. This is a section of $\gamma_i$ and since $\tiH$ is cocommutative, it is $\tiH$-linear (recall that $\AC$ is a $\tiH$-module algebra, so $\tiH$ acts diagonally on the tensor product $S(V)\tensor H\tensor C$).

Now let $y:=\eta(\gry)\in F_{i-1}$. Then $\gamma_{i-1}(y)=\gry$ and, as $\eta$ is $\tiH$-linear, $y$ lies in $(\AC)\idet$. Thus $z:=d(y)+\DeltaC(\tih)\in F_i$ lies in $P$ and
\[ \gamma_i(x-z)
 = \gamma_i(x) 
 		- \gamma_i(d(y)) - \gamma_i(\DeltaC(\tih)) 
 = \gamma_i(x) - \grd(\gry) - \gamma_i(\DeltaC(\tih)) 
 = 0
\ ,
\]
so $x-z$ lies in $F_{i-1}(A\tensor C)$.

Finally,
\[ d(x - z)
 = d(x) - d\circ\DeltaC(\tih) - d^2(y) 
 = 0
\]
using that $d(x)=0$ by assumption, $d\circ\DeltaC=0$ and $d^2=0$ by \ref{kernel-dtriv}. This means that $x-z$ is in $\ker d$ and, as $x$ and $z$ lie in $(\AC)\itriv$, their difference has to be an element of this set, too. Hence $x-z$ lies in $F_{i-1}(\ker d\itriv)$, so by the induction hypothesis, $x-z\in P$. Hence $x$ is in $P$.
\end{proof}

\section{Dirac cohomology and Vogan's conjecture}
\label{sec-Vogans-conjecture}
Finally, let us turn our attention to the cohomology of the Dirac element $D$ acting on suitable modules. To this end we fix a spin representation $S$ of $C=C(V)$ and we recall the definition of $H'=\tiH/\ker\DeltaC$ in \ref{def-tiHprime}.

Our results on the cohomology of $d\itriv$ implies the following statement which is called ``algebraic Vogan's conjecture'' in the context of \cite{HP}.
\begin{prop} \label{center-of-A} \label{algebraic-Vogans-conjecture} \label{map-zeta} Let $z$ be an element in the center of $A$. Then there is an element $a\in(\AC)\idet\iodd$ and a uniquely determined element $\zeta(z)\in Z(H')$ such that
\begin{equation} \label{formula-algebraic-Vogan}
 z\tensor 1 = \DeltaC(\zeta(z)) + Da + aD
\ . 
\end{equation}

Moreover, the map $\zeta:Z(A)\to Z(H')$ is an algebra homomorphism.
\end{prop}

\begin{proof} Since $z$ is central in $A$, $z\tensor 1$ is central in $A\tensor C$, in particular it is in $(\AC)\itriv$, and it is in the even part (with respect to the $\ZZ_2$-gradation which $\AC$ inherits from $C$). Hence $d(z\tensor 1)=[D,z\tensor 1]=0$. So by \ref{homology-of-d}, $z\tensor 1=\DeltaC(\tih)+d(a)$ for some $\tih\in Z(\tiH)$ and $a\in(\AC)\idet$, and we even know that $\DeltaC(\tih)$ is uniquely determined, so $\zeta(z)\in Z(H')$ is uniquely determined.

Now $z\tensor 1$ is in the even part of $\AC$, the map $\DeltaC$ preserves the $\ZZ_2$-gradation and the map $d$ inverts it. So we can take the even component of \eqref{formula-algebraic-Vogan} to see that $a$ can be assumed to be odd, and in particular, $d(a)=Da+aD$, which completes the proof of the first assertion.

It remains to show that $\zeta$ is an algebra homomorphism. Consider $z_1,z_2\in Z(A)$, then $z_1z_2$ is central as well and
\[ z_1z_2\tensor1 = \DeltaC(\zeta(z_1z_2))+d(a)
\]
for a uniquely determined $\zeta(z_1z_2)\in Z(H')$ and an element $a\in A\tensor C$. On the other hand,
\[ z_1z_2\tensor 1
 = (z_1\tensor 1)(z_2\tensor 1)
 = (\DeltaC(\zeta(z_1)+d(a_1)) (\DeltaC(\zeta(z_2)+d(a_2))
\]
for suitable $a_1,a_2\in(\AC)\idet$ as in the theorem. Let $\theta$ be the linear automorphism of $\AC$ which multiplies homogeneous elements $x\in\AC$ by $(-1)^{|x|}$. Then the last expression can be simplified to be
\[ \dots
 = \DeltaC(\zeta(z_1)\zeta(z_2))
  + d(  a_1\DeltaC(\zeta(z_2)) + \theta\circ\DeltaC(\zeta(z_1))a_2 
  + a_1 d(a_2)  )
 \ ,
\]
because $\DeltaC$ is an algebra map and $d$ is a graded derivation which sends $\DeltaC(\zeta(z_1)), \DeltaC(\zeta(z_2))$ and $d(a_2)$ to $0$. So by uniqueness, $\zeta(z_1z_2)=\zeta(z_1)\zeta(z_2)$ in $H'$.
\end{proof}

Let us recall the definition of Dirac cohomology (\ref{defn-Dirac-cohomology}): If $M$ is an $A$-module, then $M\tensor S$ is an $A\tensor C$-module, and in particular $D$ acts on $M\tensor S$. The Dirac cohomology was defined to be 
\[ H^D(M):=\ker D / (\im D\cap\ker D)
\ . 
\]

\begin{lem} $H^D(M)$ is an $H'$-module with an action of $H'$ induced by the action of $\tiH$ on $M\tensor S$, where $\tih\in\tiH$ acts as $\DeltaC(\tih)$.
\end{lem}

\begin{proof} To show that the action is well-defined, we have to show that the action of $\tiH$ on $M\tensor S$ preserves the kernel of $D$. But this follows from the fact that for $\ZZ_2$-homogeneous $\tih\in\tiH$, $D$ commutes with $\DeltaC(\tih)$ up to a sign (depending on $|\tih|$, \ref{commutation-D-tih}).
\end{proof}

An immediate consequence of the above results is a statement relating the central character of a module with its Dirac cohomology. It is called ``Vogan's conjecture'' in the context of \cite{HP}. 

\begin{thm}[Vogan's conjecture] \label{Vogans-conjecture} Let $M$ be an $A$-module with central character $\chi$. Suppose $H^D(M)\neq0$ and let $(U,\sigma)$ be a non-zero $H'$-submodule of $H^D(M)$.
Then $\chi = \sigma\circ\zeta$ (where $\zeta$ is the algebra map form \ref{map-zeta}).
\end{thm}

\begin{proof} We fix a central element $z\in Z(A)$. By \ref{center-of-A}, there is an element $a\in A\tensor C$ such that
\[ z\tensor 1 - \DeltaC(\zeta(z)) = Da + aD
\ . 
\]
Let $x$ be a non-zero element in $U$. Letting the two sides of the equation act on $x$ yields
\[ (\chi(z) - \sigma\circ\zeta(z)) x = (Da) \cdot x+(aD) \cdot x = 0
\ . 
\]
Now the left hand hand side is a scalar multiple of $x$ which was chosen to be non-zero, so $\chi(z) = \sigma\circ\zeta(z)$, as desired.
\end{proof}

The following result is useful in order to actually compute the Dirac cohomology for concrete modules (cp. \cite[p.~62-63]{HP-book}):

\begin{lem} \label{lem-D-diagonalizable} If $D$ acts diagonalizably on $M\tensor S$ (e.g., as a normal operator), then $H^D(M)\cong \ker D^2$.
\end{lem}

\begin{proof} In this case $\im D\cap \ker D=0$ and $\ker D=\ker D^2$.
\end{proof}

\newcommand{\arxiv}[2]{\href{http://arxiv.org/abs/#1}{arXiv:#1} [#2]}


\bigskip
\bigskip

(J. Flake) 
\textsc{Department of Mathematics, Rutgers University,
Hill Center -- Busch Campus, 110 Frelinghuysen
Road Piscataway, NJ 08854--8019, USA} \\
\texttt{flake@math.rutgers.edu}

\end{document}